\documentclass [a4paper,12pt]{amsart}
\usepackage{graphicx}
\usepackage[ansinew]{inputenc}    % accents i altres
\usepackage[all]{xy}
\usepackage{multicol}

\usepackage{amssymb,amscd}

\usepackage{enumitem}

\newtheorem{thm}{Theorem}[section]

\def\im{\operatorname{Im}}

\def\Ann{\operatorname{Ann}}

\def\coker{\operatorname{coker}}

\def\depth{\operatorname{depth}}

\def\syzygy{\operatorname{syzygy}}
\def\Hom{\operatorname{Hom}}

\def\C{\mathbb C}

\def\dim{\operatorname{dim}}

\def\coker{\operatorname{coker}}

\usepackage{graphicx}
\usepackage[ansinew]{inputenc}    % accents i altres
\usepackage[all]{xy}
\usepackage{amssymb,amscd}

\newtheorem{cor}[thm]{Corollary}
\newtheorem{teo}[thm]{Theorem}
\newtheorem{lem}[thm]{Lemma}
\newtheorem{prop}[thm]{Proposition}
\newtheorem{conj}[thm]{Conjecture}

\theoremstyle{definition}

\def\C{\mathbb C}

\def\dim{\operatorname{dim}}
\def\Tor{\operatorname{Tor}}

\def\coker{\operatorname{coker}}

\def\Eu{\operatorname{Eu}}
\def\length{\operatorname{length}}
\usepackage{color}

\thanks{The first author has been partially supported by CAPES. Grant PGC2018-094889-B-100 funded by MCIN/AEI/ 10.13039/501100011033 and by ``ERDF A way of making Europe'' (second author). The third and fourth author have been partially supported by  FAPESP Grant 2019/07316-0.}
\keywords{isolated complete intersection singularity, Bruce-Roberts number, relative Bruce-Roberts number, logarithmic characteristic variety}
\subjclass[2000]{Primary 32S25; Secondary 58K40, 32S50}

 \everymath{\displaystyle}
\begin{document}

\title[The Bruce-Roberts Numbers of a Function on an ICIS]{The Bruce-Roberts Numbers of a Function on an ICIS}

\author{B. K. Lima-Pereira, J.J. Nu\~no-Ballesteros, B. Or\'efice-Okamoto, J.N. Tomazella}

\address{Departamento de Matem\' atica, Universidade Federal de S\~ao Carlos, Caixa Postal 676, 13560-905, S\~ao Carlos, SP, BRAZIL}
\email{barbarapereira@estudante.ufscar.br}

\address{Departament de Matem\`atiques,
Universitat de Val\`encia, Campus de Burjassot, 46100 Burjassot
SPAIN. \newline Departamento de Matemática, Universidade Federal da Paraíba
		CEP 58051-900, João Pessoa - PB, BRAZIL}
\email{Juan.Nuno@uv.es}

\address{Departamento de Matem\'atica, Universidade Federal de S\~ao Carlos, Caixa Postal 676,
	13560-905, S\~ao Carlos, SP, BRAZIL}
\email{brunaorefice@ufscar.br}

\address{Departamento de Matem\'atica, Universidade Federal de S\~ao Carlos, Caixa Postal 676,
	13560-905, S\~ao Carlos, SP, BRAZIL}

\email{jntomazella@ufscar.br}

\maketitle

%\begin{abstract}
%
%Given $(X,0)\subset(\C^n,0)$ a germ of hypersurface with isolated singularity and $\f$ a function germ with $\mu_{BR}(f,X)<\infty $, we show that $\mu_{BR}(f,X)=\mu(f)+\mu(\phi,f)+\mu(X)-\tau(X)$, where $\mu_{BR}(f,X)$ is the Bruce-Roberts number of $f$ with respect to $X$,  $\mu(f)$ is the Milnor number of $f$, $\mu(\phi,f)$ is the Milnor number of  $X\cap f^{-1}(0)$, $\mu(X)$ is the Milnor number and $\tau(X)$ is the Tjurina number of $(X,0)$. 
%\end{abstract}

%\begin{abstract}
%In this paper we generalize the relation given in \cite{Orefice}. We considered $(X,0)\subset(\C^{n},0)$ a germ of hypersurface with isolated singularity and $f\colon(\C^{n},0)\to(\C,0)$ a finitely $\mathcal{R}_{X}$-determined germ of function germ. We use only techniques of singularity theory for to proof $\mu_{BR}(f,X)=\mu(f)+\mu(\phi,f)+\mu(X)-\tau(X)$. Besides we get to a beautiful characterization for Tjurina number of $(X,0)$, and a easier form to calculate the Bruce-Roberts number of $f$ with respect to $(X,0)$.
%\end{abstract} 
%
%\begin{abstract}
%Let $(X,0)$ be an ICIS defined by $\phi=(\phi_{1},...,\phi_{k})\colon(\C^{n},0)\to(\C^{k},0)$ we consider $f\colon(\C^n,0)\to\C$ such that the relative Bruce-Roberts number $\mu_{BR}^{-}(f,X)$ is finite and we generalize \cite[Proposition 7.7]{bruce roberts} proving that $\mu_{BR}^{-}(f,X)=\mu(f^{-1}(0)\cap X,0)+\mu(X,0)-\tau(X,0)$ and as a consequence the relative logarithmic characteristic variety $LC(X,0)^{-}$ is Cohen-Macaulay. 
%\end{abstract}

\begin{abstract}

We give formulas for the Bruce-Roberts number $\mu_{BR}(f,X)$ and its relative version $\mu_{BR}^{-}(f,X)$ of a function $f$ with respect to an ICIS $(X,0)$. We show that $\mu_{BR}^{-}(f,X)=\mu(f^{-1}(0)\cap X,0)+\mu(X,0)-\tau(X,0)$, where $\mu$ and $\tau$ are the Milnor and Tjurina numbers, respectively, of the ICIS.
The formula for $\mu_{BR}(f,X)$ is more complicated and also involves $\mu(f)$ and some lengths in terms of the ideals $I_X$ and $Jf$. 
We also consider the logarithmic characteristic variety, $LC(X)$, and its relative version, $LC(X)^{-}$. 
We show that $LC(X)^{-}$ is Cohen-Macaulay and that $LC(X)$ is Cohen-Macaulay at any point not in $X\times\{0\}$. We generalize previous results presented by the authors when $(X,0)$ has codimension one  and by Bruce and Roberts when it is weighted homogeneous of any codimension. 
\end{abstract}

\section{Introduction}

An important invariant of the germ of an analytic function $f : (\C^{n}, 0) \to (\C, 0)$ is its Milnor number,
$\mu(f)$, which is equal to $\dim_{\C}\mathcal{O}_{n}/Jf$, where $\mathcal{O}_{n}$ is the ring of germs of analytic functions on $(\C^n,0)$, and $Jf=\langle\partial f/\partial x_{i}\rangle$ is the Jacobian ideal of $f$.%, generated in $\mathcal{O}_{n}$ by its partial derivatives, $\partial f/\partial xi.$Let $\mathcal{O}_{n}$ the local ring of of holomorphic functions germs $(\C^{n},0)\to\C$, the Milnor number of $f\in\mathcal{O}_{n}$ is defined

Let $(X, 0)\subset(\C^{n}, 0)$ be the germ of an analytic variety and let %. Bruce and Roberts, in \cite{bruce roberts}, introduced the (relative) Bruce-Roberts numbers, which are generalizations of $\mu(f)$.
% which we will call the Bruce-Roberts number and denote by $\mu_{BR}(f,X)$. Like
%the Milnor number of $f$ , this number shows some geometric properties of $f$ and $X$. For instance, if
%one consider the group $\mathcal{R}_{X}$ of analytic automorphisms of $(\C^{n}, 0)$ which preserve $X$, then $f$ is finitely
%determined with respect to the action of $\mathcal{R}_{X}$ on $\mathcal{O}_{n}$ if and only if $\mu_{BR}(f,X)$ is finite.
 $\Theta_{X}$ be the $\mathcal{O}_{n}$-submodule of vectors fields that are tangent to $(X,0)$, that is, $$\Theta_{X}=\{\xi\in\Theta_{n}|\;dh(\xi)\in I_{X},\;\forall h\in I_{X}\},$$
where $\Theta_{n}$ is the $\mathcal O_n$-module of germs of vector fields on $(\C^n,0)$ and $I_{X}\subset \mathcal{O}_{n}$ is the ideal that defines $(X,0).$ %We denote by $df(\Theta_{X})$ the ideal in $\mathcal{O}_{n}$ generated by $\xi(f):=df(\xi)$ such that $\xi\in\Theta_{X}$ and $df$ is the differential of $f$.
The Bruce-Roberts number and the relative Bruce-Roberts number are, respectively, $$\mu_{BR}(f,X)=\dim_{\C}\frac{\mathcal{O}_{n}}{df(\Theta_{X})},\;\mu_{BR}^{-}(f,X)=\dim_{\C}\frac{\mathcal{O}_{n}}{df(\Theta_{X})+I_{X}},$$ where $df(\Theta_{X})$ is the image of $\Theta_{X}$ by the differential of $f$. %  which we call the Bruce-Roberts number and the relative Bruce-Roberts number, respectively.
These numbers are defined in \cite{bruce roberts} and may be considered as generalizations of the Milnor number of the function germ, because when $X=\C^{n}$ then $\Theta_{X}=\Theta_n$ and $df(\Theta_{X})=Jf$. 

 In general, the computation of both invariants is not easy since the submodule $\Theta_{X}$ is a complicated object and usually it requires the use of a computer algebra system like {\sc Singular} \cite{singular}. So, it is interesting to obtain formulas which give $\mu_{BR}(f,X)$ or $\mu_{BR}^-(f,X)$ in terms of other known invariants. The case where $(X,0)$ is an isolated hypersurface singularity IHS was considered previously in \cite{nunoballesteros oreficeokamoto limapereira tomazela, segundo artigo, Orefice}. In this paper we extend the formulas to the case that $(X,0)$ is an isolated complete intersection singularity ICIS. Our main results are
\begin{align*}
\mu_{BR}^{-}(f,X)&=\mu(X\cap f^{-1}(0),0)-\mu(X,0)+\tau(X,0),\\
\mu_{BR}(f,X)&=\mu_{BR}^{-}(f,X)+\mu(f)-\dim_{\C}\frac{\mathcal{O}_{n}}{Jf+I_{X}}+\dim_{\C}\frac{I_{X}\cap Jf}{I_{X}Jf},
\end{align*}
where $\mu$ and $\tau$ are the Milnor and Tjurina numbers respectively of an ICIS. We remark that both formulas extend the ones obtained in \cite{nunoballesteros oreficeokamoto limapereira tomazela,segundo artigo,  Orefice} when $(X.0)$ is an IHS and that the formula for $\mu_{BR}^{-}(f,X)$ also appears in \cite{bruce roberts} in the particular case that $(X,0)$ is a weighted homogeneous ICIS.

Another important property is that,
like the Milnor number, the Bruce-Roberts numbers $\mu_{BR}(f,X)$ and $\mu_{BR}^{-}(f,X)$ may be calculated in terms of 
the number of stratified critical points of a Morsification of $f$ with respect to the logarithmic stratification of $X$. This happens when 
the logarithmic characteristic variety $LC(X)$ and its relative version $LC(X)^-$, respectively, are Cohen-Macaulay. In fact, the Cohen-Macaulayness of $LC(X)$ and $LC(X)^-$ implies the conservation of both numbers in any deformation of $f$. Many authors have recent papers about these issues \cite{tomazellaruas, carlesruas,  Nivaldo, Grego, nunoballesteros oreficeokamoto limapereira tomazela, segundo artigo, Orefice, Tajima}.

Here we show that if $(X,0)$ is any ICIS, then $LC(X)^-$ is Cohen-Macaulay and $LC(X)$ is also Cohen-Macaulay at any point not in $X\times\{0\}$. Again, this extends previous results of \cite{segundo artigo, nunoballesteros oreficeokamoto limapereira tomazela, Orefice} when $(X,0)$ is an IHS and of \cite{bruce roberts} when $(X,0)$ is a weighted homogeneous ICIS. We remark that when $(X,0)$ has codimension $>1$, $LC(X)$ is not Cohen-Macaulay at any point in $X\times\{0\}$ (see \cite[5.10]{bruce roberts}).

As a byproduct of our process, we also prove that the Tjurina number $\tau(X,0)$ of an ICIS $(X,0)$ can be computed as
\[
\tau(X,0)=\dim_\C\frac{\Theta_X}{\Theta_X^T},
\]
where $\Theta_X^T$ is the submodule of $\Theta_X$ of trivial vector fields. This was proved in \cite{nunoballesteros oreficeokamoto limapereira tomazela, Tajima} for IHS.

\section{The relative Bruce-Roberts number}

%Through this section $(X,0)$ is the ICIS determined by the analytic map germ $\phi=(\phi_{1},...,\phi_{k}):(\C^{n},0)\to(\C^{k},0)$ and $f:(\C^{n},0)\to(\C,0)$ is an analytic function germ.

  % This result is not true if $(X,0)$ is not weighted homogeneous but,
%We show that if $(X,0)$ is an isolated hypersurface singularity, IHS, we need to add the difference  
% In order to prove the results in \cite{segundo artigo} we present there a characterization of the Tjurina number of IHS in terms of its tangent vector fields, the next lemma presents this characterization for any ICIS. 
When $(X,0)$ is a weighted homogeneous ICIS, the generators of $\Theta_{X}$ are exhibited in \cite{wahl}. For the non weighted homogeneous case we resort to the trivial vector fields instead, which are defined as
 % In the proofs of the results which relates the Bruce-Roberts number to other knew invariants in singularity theory \cite{bruce roberts, Orefice}, it was crucial to know a set of generators for $\Theta_X$ which was provided by \cite{wahl} when $(X,0)$ is a weighted homogeneous ICIS.  For the not weighted homogeneous case, we will use the module of the trivial vector fields 
$$\Theta_{X}^{T}=\langle \xi\in\Theta_{X};\;d\phi_{i}(\xi)=0,\;\forall i=1,...,k\rangle+\langle\phi_{j}\partial/\partial x_{i}\;j=1,...,k;i=1,...,n\rangle.$$ %The next proposition presents a set of generators for it. This result will substitute the one of Wahl for this case. %In \cite{wahl},  Wahl  exhibits the generators of $\Theta_{X}$ when $(X,0)$ is weighted homogeneous. Up to now, we do not know a similar result for not weighted homogeneous varieties but we can look for the trivial vector fields  %When $(X,0)$ is an IHS it is not hard to see that 
% $$\Theta_{X}^{T}=I_{2}\begin{pmatrix}
%\tfrac{\partial}{\partial x_{1}}&...&\tfrac{\partial}{\partial x_{n}}\vspace{0.1cm}\\
%\tfrac{\partial\phi}{\partial x_{1}}&...&\tfrac{\partial\phi}{\partial x_{n}}
%\end{pmatrix}+\left\langle\phi\tfrac{\partial}{\partial x_{j}},\;j=1,...,n\right\rangle,$$
%where $I_{k}$ denotes the ideal generate by the minors of order $k$ of the matrix. In the next proposition we generalize this result for any ICIS.
% characterization.% of the generators of $\Theta_{X}^{T}$, and it is fundamental to prove our principal result. 
% when $(X,0)$ is an IHS, the partial derivatives of $\phi_{1}$ is a regular sequence then in this case $$\Theta_{X}^{T}=I_{2}\begin{pmatrix}
%\tfrac{\partial}{\partial x_{1}}&...&\tfrac{\partial}{\partial x_{n}}\vspace{0.1cm}\\
%\tfrac{\partial\phi_{1}}{\partial x_{1}}&...&\tfrac{\partial\phi_{1}}{\partial x_{n}}
%%\end{pmatrix}+\left\langle\phi_{1}\frac{\partial}{\partial x_{j}},\;j=1,...,n\right\rangle.$$
%We prove a similar characterization to generators of $\Theta_{X}^{T}$ when $(X,0)$ is any ICIS, and we use the generalized koszul complex introduced by Buchsbaum and Rim \cite{buchsbaum rim}.

\begin{prop}\label{prop:trivial} Let $(X,0)$ be the  ICIS determined by $\phi=(\phi_{1},...,\phi_{k}):(\C^{n},0)\to(\C^{k},0)$, then 
$$\Theta_{X}^{T}=I_{k+1}\begin{pmatrix}
\tfrac{\partial}{\partial x_{1}}&\hdots&\tfrac{\partial}{\partial x_{n}}\vspace{0.1cm}\\
\tfrac{\partial\phi_{1}}{\partial x_{1}}&\hdots&\tfrac{\partial\phi_{1}}{\partial x_{n}}\\
\vdots&\ddots&\vdots\\
\tfrac{\partial\phi_{k}}{\partial x_{1}}&\hdots&\tfrac{\partial\phi_{k}}{\partial x_{n}}\\
\end{pmatrix}+\left\langle\phi_{i}\tfrac{\partial}{\partial x_{j}},\; i=1,...,k,\;j=1,...,n\right\rangle,$$
where the first term in the right hand side is the submodule of $\Theta_X$ generated by the $(k+1)$-minors of the matrix.
\end{prop}

To prove the proposition we use the generalized Koszul complex introduced by Buchsbaum and Rim \cite{buchsbaum rim}.
Let  $R$ be a commutative Noetherian  ring, $g:R^{m}\to R^{l}$ an $R$-homomorphism and
\begin{align*}
\gamma(g):R^{m}\times {R^{l}}^{*}&\to R\\
(b,a)&\mapsto a(g(b))
\end{align*}
 with ${R^{l}}^{*}=\Hom(R^{l},R)$.% o conjunto dos homomorfismos de $R^{n}$ em $R$.  

The generalized Koszul complex, $K(\bigwedge^{p}g)$, for each $p$, is defined by:
\begin{equation}\label{complexo generalizado}
\hdots\longrightarrow\sum_{s_{0}\geqslant l+1-p}\bigwedge^{s_{0}}{R^{l}}^{*}\otimes\bigwedge^{s_{1}}{R^{l}}^{*}\otimes\bigwedge^{p+\sum s_{i}}R^{m}\longrightarrow\sum_{s_{0}\geqslant l+1-p}\bigwedge^{s_{0}}{R^{l}}^{*}\otimes\bigwedge^{p+s_{0}}R^{m}\stackrel{d}\longrightarrow\bigwedge^{p}R^{m}\stackrel{\bigwedge^{p}g}\longrightarrow\bigwedge^{p}R^{l}
\end{equation}
with $s_{i}\geq 1$ for all $i\geq 1$. The differential $$d:\sum_{s_{0}\geqslant l+1-p}\bigwedge^{s_{0}}{R^{l}}^{*}\otimes\bigwedge^{p+s_{0}}R^{m}\to\bigwedge^{p}R^{m}$$
is defined by:
$$d(\alpha\otimes\beta)%=\omega_{\alpha}(\beta)
=\sum_{1\leq j_{1}<...<j_{p}\leq p+{1}}(-1)^{\sum j_{k}}\det(\gamma(a_{i},b_{j_{k}}))b_{1}\wedge...\wedge\hat{b_{j_{1}}}\wedge...\wedge\hat{b_{j_{p}}}\wedge...\wedge b_{p+s_{0}},$$
where $\alpha=a_{1}\wedge...\wedge a_{s_{0}}\in\bigwedge^{s_{0}}{R^{l}}^{*}$ and $\beta=b_{1}\wedge...\wedge b_{p+s_{0}}\in\bigwedge^{p+s_{0}}R^{m}$.

%and we obtain the map $$\omega_{\alpha}:\bigwedge^{s_{0}}{R^{l}}^{*}\otimes\bigwedge^{p+s_{0}}R^{m}\to\bigwedge^{p}R^{m}.$$

If $\coker(g)\neq 0$, $K(\bigwedge^{p}g)$ is a free resolution of $\coker(\bigwedge^{p}g)$ for some $p$, $1\leq p\leq l$ (or for all $p$, $1\leq p\leq l$) if and only if $\depth(I(g),R)=m-l+1$, where $I(g)$ is the annihilator of $\coker(\bigwedge^{l}g)$ (\cite[Corollary 2.6]{buchsbaum rim}).

\begin{proof}[Proof of Proposition \ref{prop:trivial}]

We consider the previous complex with the following homomorphism  of $\mathcal{O}_{n}$-modules
$$d\phi=\begin{pmatrix}
\tfrac{\partial\phi_{1}}{\partial x_{1}}&\hdots&\tfrac{\partial\phi_{1}}{\partial x_{n}}\\
\vdots&\ddots&\vdots\\
\tfrac{\partial\phi_{k}}{\partial x_{1}}&\hdots&\tfrac{\partial\phi_{k}}{\partial x_{n}}
\end{pmatrix}:\mathcal{O}_{n}^{n}\to\mathcal{O}_{n}^{k}.$$
Here $I(d\phi)$ is the annihilator of $\coker(\bigwedge^{k}d\phi)=\mathcal{O}_{n}/J(\phi_{1},...,\phi_{k})$, that is,  $I(d\phi)$ is the ideal generated by the minors of maximum order of the Jacobian matrix of $\phi=(\phi_{1},...,\phi_{k})$ and $$\dim\frac{\mathcal{O}_{n}}{I(d\phi)}=k-1=n-(n-k+1)(k-k+1).$$ 
Therefore, $K(\bigwedge^{p}d\phi)$ is a free resolution of $\coker(\bigwedge^{p}d\phi)$, for all $p$, $1\leq p\leq k.$

Considering $p=1$, $K(\bigwedge^{1}d\phi)=K(d\phi)$ is an exact sequence and the final part of (\ref{complexo generalizado}) in this case is equal to
$$\bigwedge^{k}{{\mathcal{O}_{n}}^{k}}^{*}\otimes\bigwedge^{1+k}\mathcal{O}_{n}^{n}\stackrel{d}\longrightarrow\mathcal{O}_{n}^{n}\stackrel{d\phi}\longrightarrow\mathcal{O}_{n}^{k}, 
$$ and $$\im(d)=I_{k+1}\begin{pmatrix}
\tfrac{\partial}{\partial x_{1}}&...&\tfrac{\partial}{\partial x_{n}}\vspace{0.1cm}\\
\tfrac{\partial\phi_{1}}{\partial x_{1}}&\hdots&\tfrac{\partial\phi_{1}}{\partial x_{n}}\\
\vdots&\ddots&\vdots\\
\tfrac{\partial\phi_{k}}{\partial x_{1}}&\hdots&\tfrac{\partial\phi_{k}}{\partial x_{n}}\\
\end{pmatrix}=\ker(d\phi).$$
\end{proof}

%The previous result generalizes the characterization of $\Theta_{X}^{T}$ when $(X,0)$ is an IHS.
 As a consequence of the previous proposition, for any function germ $f\colon(\C^{n},0)\to\C$, 
 \begin{equation}\label{dfthetaxt}df(\Theta_{X}^{T})=J(f,\phi)+\langle\phi_{j}\tfrac{\partial f}{\partial x_{i}}\;i=1,...,n;\;j=1,...,k\rangle,\end{equation} where $J(f,\phi)$ is the ideal in $\mathcal{O}_{n}$ generated by the maximal minors of the Jacobian matrix of $(f,\phi_{1},...,\phi_{k})$.
\begin{teo}\label{relation}
Let $(X,0)\subset (\C^{n},0)$ be an ICIS and $f\in\mathcal{O}_{n}$ such that $\mu_{BR}^{-}(f,X)<\infty$, then $(X\cap f^{-1}(0),0)$ defines an ICIS and $$\mu(X\cap f^{-1}(0),0)=\mu_{BR}^{-}(f,X)-\mu(X,0)+\tau(X,0).$$
\end{teo}
\begin{proof} Let $\phi=(\phi_{1},...,\phi_{k}):(\C^{n},0)\to(\C^{k},0)$ be the map which defines $(X,0)$. From the equality (\ref{dfthetaxt}) and the exact sequence $$0\longrightarrow\frac{df(\Theta_{X})+I_{X}}{df(\Theta_{X}^{T})+I_{X}}\longrightarrow \frac{\mathcal{O}_{n}}{df(\Theta_{X}^{T})+I_{X}}\longrightarrow\frac{\mathcal{O}_{n}}{df(\Theta_{X})+I_{X}}\longrightarrow 0,$$ we have $$\mu(X\cap f^{-1}(0),0)=\mu_{BR}^{-}(f,X)+\dim_{\C}\frac{df(\Theta_{X})+I_{X}}{df(\Theta_{X}^{T})+I_{X}}-\mu(X,0).$$ Then we need to prove  $$\dim_{\C}\frac{df(\Theta_{X})+I_{X}}{df(\Theta_{X}^{T})+I_{X}}=\tau(X,0)=\dim_{\C}\frac{\mathcal{O}_{n}^{k}}{\im d\phi+I_{X}\mathcal{O}_{n}^{k}},$$ where the second equality is a characterization for Tjurina number, (see \cite[Theorem 1.16]{greuel deformations}). % and $\im d\phi$ is the submodule of $\mathcal{O}_{n}^{k}$ generated by the columns of jacobian matrix of $\phi$,(see \cite[Theorem 1.16]{greuel deformations}). 
  %Then by the Lemma \ref{first equality} we need to prove that $\dim_{\C}\mathcal{O}_{n}^{k}/(\im d\phi+I_{X}\mathcal{O}_{n}^{k})=\dim_{\C}(df(\Theta_{X})+I_{X})/(df(\Theta_{X}^{T})+I_{X})$.  %have that $$\mu_{BR}^{-}(f,X)=\mu(X\cap f^{-1}(0),0)+\mu(X,0)-\dim_{\C}\frac{df(\Theta_{X})+I_{X}}{df(\Theta_{X}^{T})+I_{X}}.$$ 
%We consider the algebraic characterization of Tjurina number, \cite{greuel  deformations},  $$\tau(X,0)=\dim_{\C}\frac{\mathcal{O}_{n}^{k}}{\im d\phi+I_{X}\mathcal{O}_{n}^{k}},$$ then by the Lemma \ref{first equality} we need to prove the following equality $$\dim_{\C}\frac{df(\Theta_{X})+I_{X}}{df(\Theta_{X}^{T})+I_{X}}=\dim_{\C}\frac{\mathcal{O}_{n}^{k}}{\im d\phi+I_{X}\mathcal{O}_{n}^{k}},$$ for any function germ $f$ such that $(X\cap f^{-1}(0),0)$ defines an ICIS. 
%\begin{teo}\label{tjurina number}
%Let $(X,0)$ be an ICIS determined by $\phi=(\phi_{1},...,\phi_{k})$ and $f\in\mathcal{O}_{n}$ a function germ such that $(X\cap f^{-1}(0),0)$ defines an ICIS, then $$\dim_{\C}\frac{df(\Theta_{X})+I_{X}}{df(\Theta_{X}^{T})+I_{X}}=\tau(X,0).$$
%\end{teo}
%\begin{proof}
Let us consider the following exact sequence
$$0\longrightarrow\ker(\overline{\alpha})\stackrel{i}\longrightarrow\frac{\mathcal{O}_{n}}{df(\Theta_{X}^{T})+I_{X}}\stackrel{\overline{\alpha}}\longrightarrow\frac{\mathcal{O}_{n}^{k+1}}{\im d(f,\phi)+I_{X}\mathcal{O}_{n}^{k}}\stackrel{\overline{\pi}}\longrightarrow\frac{\mathcal{O}_{n}^{k}}{\im d(\phi)+I_{X}\mathcal{O}_{n}^{k}}\longrightarrow 0.$$
in which  $i$ is the inclusion and $\overline{\pi}$ and $\overline{\alpha}$ (respectively) are induced by 
$$\pi:\mathcal{O}_{n}^{k+1}\to\mathcal{O}_{n}^{k}\textup{ and }\alpha:\mathcal{O}_{n}\to\mathcal{O}_{n}^{k+1}$$ given by $\pi(a_{1},a_{2},...,a_{k})=(a_{2},...,a_{k})$ and $\alpha(a)=(a,0,...,0).$ %We observe that all these maps are well defined, $i$ and $\overline{\pi}$ is immediate and if $a\in df(\Theta_{X}^{T})+I_{X}=J(f,\phi)+I_{X}$ then exist a vector field $\xi\in\Theta_{X}^{T}$ and $\gamma\in I_{X}$ such that $$\alpha(a)=(a,0,...,0)=(df(\xi)+\gamma,0,...,0)\textup{ and }$$ $$(df(\xi)+\gamma,0,...,0)\equiv(df(\xi)+\gamma,d\phi_{1}(\xi),...,d\phi_{k}(\xi)) \mod (\im d(f,\phi)+I_{X}).$$ 
%The sequence is exact. Indeed  $$\im\overline{\alpha}=\frac{\langle(1,...,0)\rangle+\im d(f,\phi)+I_{X}\mathcal{O}_{n}^{k}}{\im d(f,\phi)+I_{X}\mathcal{O}_{n}^{k}}=\ker\overline{\pi}.$$ 
 
Since this ring  $\mathcal{O}_{n}/I_{X}$ is Cohen-Macaulay and the Jacobian matrix of $(f,\phi)$ is a parameter matrix in the sense of \cite{buchsbaum rim} for this ring, we have  $$\dim_{\C}\frac{\mathcal{O}_{n}}{df(\Theta_{X}^{T})+I_{X}}=\dim_{\C}\frac{\mathcal{O}_{n}^{k+1}}{\im(d(f,\phi))+I_{X}\mathcal{O}_{n}^{k+1}},$$ see \cite[p. 224]{buchsbaum rim}. Hence $\dim_{\C}\ker\overline{\alpha}=\tau(X,0)$ 
and we claim that $$\ker(\overline{\alpha})=\frac{df(\Theta_{X})+I_{X}}{df(\Theta_{X}^{T})+I_{X}}.$$
Let $\overline{a}\in\ker(\overline{\alpha})$, therefore there exists $\xi\in\Theta_{n}$ such that $$(a,0,...,0)=(df(\xi),d\phi_{1}(\xi),...,d\phi_{k}(\xi))+(\alpha_{1},...,\alpha_{k+1}),$$ with $\alpha_{i}\in I_{X}\;\forall\;i=1,...,k+1$. Thus $a-\alpha_{1}=df(\xi)$ and hence $a\in df(\Theta_{X})+I_{X}.$ %It proves the inclusion $$\ker\overline{\alpha}\subset\frac{df(\Theta_{X})+I_{X}}{df(\Theta_{X}^{T})+I_{X}}.$$

For the other inclusion, if $\xi\in\Theta_{X}$, then $\overline{df(\xi)}\in(df(\Theta_{X})+I_{X})/(df(\Theta_{X}^{T})+I_{X})$ and  $$\overline{\alpha}(\overline{df(\xi)})=\overline{(df(\xi),0,...0)}=\overline{(df(\xi),d\phi_{1}(\xi),...,d\phi_{k}(\xi))}\in\im(d(f,\phi)).$$%then we prove the other inclusion. and we have the equality  $$\ker\overline{\alpha}=\frac{df(\Theta_{X})+I_{X}}{df(\Theta_{X}^{T})+I_{X}}.$$ 
\end{proof}

%The previous result proves an interesting characterization for Tjurina number of an ICIS $(X,0)$, 
It is curious that by the previous result the dimension of the quotient $(df(\Theta_{X})+I_{X})/(df(\Theta_{X}^{T})+I_{X})$ as a $\C$-vector space does not depend of the function germ $f$ such that $(X\cap f^{-1}(0),0)$ defines an ICIS. 

%\begin{teo}\label{relation}
%Let $(X,0)\subset (\C^{n},0)$ be an ICIS and $f\in\mathcal{O}_{n}$ such that $\mu_{BR}^{-}(f,X)<\infty$, then $(X\cap f^{-1}(0),0)$ defines an ICIS and $$\mu(X\cap f^{-1}(0),0)=\mu_{BR}^{-}(f,X)-\mu(X,0)+\tau(X,0).$$
%\end{teo}
%\begin{proof}
%It follows from the Lemma \ref{first equality} and Lemma \ref{caracterizacao do tjurina para usar no relativo}.
%\end{proof}

\begin{cor}\label{corteo1}
Let $(X,0)$ be a weighted homogeneous ICIS, then $$\mu(X,0)=\tau(X,0).$$
\end{cor}
\begin{proof}
It follows of Theorem \ref{relation} and \cite[Proposition 7.7]{bruce roberts}.
\end{proof}
\begin{cor}\label{corteo2primeira parte}
The relative Bruce-Roberts number is a topological invariant for a family of functions germs over a fixed ICIS.
\end{cor}
\begin{cor}\label{corteo2}
 If $(X,0)$ is an ICIS and $f$ is an $\mathcal{R}_{X}$-finitely determined function germ, then
$\mu_{BR}^{-}(f,X)=\dim_{\C}\mathcal{O}_{n}/(J(f,\phi)+I_{X})-\tau(X,0).$
\end{cor}

%We use our results to relate the relative Bruce-Roberts number of a generic linear projection to the top polar multiplicity defined by Gaffney \cite{Gaffney}.
% Let $(X,0)$ be an ICIS defined by $\phi:(\C^{n},0)\to(\C^{k},0)$, Gaffney defines in \cite{Gaffney} the $(n-k)$-th polar multiplicity by $$m_{n-k}(X,0)=\dim_{\C}\frac{\mathcal{O}_{n}}{J(p,\phi)+I_{X}}$$ where $p:(\C^{n},0)\to\C$ is a linear generic projection.
From the previous corollary, if $(X,0)$ is an ICIS of codimension $k$ in $\C^n$,
 $$\mu_{BR}^{-}(p,X)=m_{n-k}(X,0)-\tau(X,0),$$ where $m_{n-k}$ is the $(n-k)$th polar multiplicity as defined in \cite{Gaffney} and $p:\C^n\to\C$ is a generic linear projection.
Hence $$m_{n-k}(X,0)=\mu_{BR}(p,X)+\tau(X,0)=\mu_{BR}^{-}(p,X)+\tau(X,0),$$ because $\mu_{BR}(p,X)=\mu_{BR}^{-}(p,X).\;$ 

%ou seja o nÃºmero de Bruce-Roberts independe da projeÃ§Ã£o linear genÃ©rica.
Moreover, proceeding as in \cite{nunoballesteros oreficeokamoto limapereira tomazela},  by  \cite{perez saia} and \cite{le tessier},  
$$(-1)^{n-k-1}-\mu(X,0)=\Eu(X,0)-m_{n-k}(X,0),$$ where $\Eu(X,0)$ is the local Euler obstruction of $(X,0)$. 
\begin{cor}
Let $(X,0)$ be an ICIS of codimension $k$ and $p:\C^{n}\to\C$ a generic linear projection, then
\begin{enumerate}
\item[a)] $m_{n-k}(X,0)=\mu_{BR}^{-}(p,X)+\tau(X,0)$;
\item[b)] $\Eu(X,0)=\mu_{BR}^{-}(p,X)+\tau(X,0)-\mu(X,0)+(-1)^{n-k-1}.$
\end{enumerate}
\end{cor}

\section{The Relative Logarithmic Characteristic Variety}

We recall the definition of the logarithmic characteristic variety of $(X,0)$. Assume that the module $\Theta_X$ is generated over $\mathcal O_n$ by vector fields $\xi_1,\dots,\xi_k$ and put
\[
\xi_i=\sum_{j=1}^n a_{ij} \frac\partial{\partial x_j}, \; a_{ij}\in\mathcal O_n.
\] 
We denote by $x_1,\dots,x_n,p_1,\dots,p_n$ the coordinates of the cotangent bundle $T^*\C^n$. For each $i=1,\dots,k$ we define
$\hat \xi_i=\sum_{j=1}^n a_{ij} p_j\in\mathcal O_n[p_1,\dots,p_n]$. The \emph{logarithmic characteristic variety} $LC(X)$ is the complex subspace of $T^*\C^n$ given by the ideal $I$ in $\mathcal O_n[p_1,\dots,p_n]$ generated by $\hat\xi_1,\dots,\hat\xi_k$. Observe that $LC(X)$ is a germ in $T^*\C^n$ along the subset $T_0^*\C^n$. The definition of $LC(X)$ does not depend on the choice of generators of $\Theta_X$ (see \cite{bruce roberts}).

Another property of $LC(X)$ is that $\dim LC(X)=n$ if and only if $(X,0)$ is holonomic (see \cite[Proposition 1.14]{bruce roberts}). In fact, if $X_0,\dots,X_r$ are the logarithmic strata, then, as set germs, 
\[
LC(X)=\bigcup_{i=0}^r \overline{N^{*}X_{i}},
\]
where $\overline{N^{*}X_{i}}$ is the closure of the conormal bundle $N^*X_i$ of $X_i$ in $T^*\C^n$. In particular, if we assume $X_{0}=\C^{n}\setminus X$, then $\overline{N^{*}X_{0}}=\C^n\times\{0\}$, the zero section given by the ideal $P=\langle p_1,\dots,p_n\rangle$. Obviously $I\subseteq P$ and the \emph{relative logarithmic characteristic variety} $LC(X)^-$ is the complex subspace of $T^*\C^n$ given by the quotient ideal $I\colon P$. As set germs,
\[
LC(X)^-=\bigcup_{i=1}^r \overline{N^{*}X_{i}}=LC(X)\cap \pi^{-1}(X),
\]
where $\pi:T^*\C^n\to\C^n$ is the projection.

In \cite[Proposition 5.10]{bruce roberts} they show that if $(X,0)$ has codimension $>1$ in $(\C^n,0)$, then $LC(X)$ is not Cohen-Macaulay at any point in $X\times\{0\}\subset LC(X)$. However, they also show in \cite[Proposition 7.3]{bruce roberts} that if $(X,0)$ is is a weighted homogeneous ICIS then $LC(X)^-$ is Cohen-Macaulay at any point and $LC(X)$ is Cohen-Macaulay at any point not in $X\times\{0\}$. We extend both results for any ICIS, non necessarily weighted homogeneous.

%Let $(X,0)$ be an ICIS and let $C_{1},...,C_{s}$ be the irreducible components of $X$. The logarithmic stratification of $X$ is given by $X_{0}=\C^{n}\setminus X,\;X_{i}=C_{i}\setminus\{0\},\;i=1,...,s\;$ and $X_{s+1}=\{0\}$ then $Y_{i}=\overline{N^{*}X_{i}},\; i=0,...,s+1$ are the irreducible components of  the logarithmic characteristic variety of $(X,0)$, $LC(X)$. 
%
%If $LC(X,0)$ is Cohen-Macaulay then the Bruce-Roberts number of a function germ with respect to $(X,0)$ is related to the multiplicities of the strata (see \cite{bruce roberts}). In the particular case when $(X,0)$ is an IHS, $LC(X)$ is Cohen-Macaulay, see \cite{nunoballesteros oreficeokamoto limapereira tomazela, Orefice}. Bruce and Roberts prove in \cite{bruce roberts} that this does not happen if the ICIS has codimension higher than one. Then they consider the closure of the set obtained after deleting the component $Y_{0}$, $LC(X)^{-}=\cup_{i=1}^{s+1}Y_{i}.$ 
%We call $LC(X)^{-}$ the relative logarithmic characteristic variety of $(X,0)$. $LC(X)^{-}$ can be Cohen-Macaulay even when $LC(X)$ is not. More precisely, if $(X,0)$ is a weighted homogeneous ICIS then $LC(X)^{-}$ is Cohen-Macaulay and $LC(X)$ is not (see \cite[Propositions 5.10 and 7.3]{bruce roberts}). In the next theorem we extend this result for not necessarily weighted homogeneous ICIS

\begin{teo}\label{LC(X)menos cm}
If $(X,0)$ is an ICIS, then $LC(X)^{-}$ is Cohen-Macaulay.
\end{teo}
\begin{proof}
This proof is motivated by \cite[Theorem 5.4]{nunoballesteros oreficeokamoto limapereira tomazela}.
We suppose that $(X,0)$ is determined by $\phi:(\C^{n},0)\to(\C^{k},0)$. Let $(0,p)\in LC(X)^{-}.$ Since $LC(X)^{-}\subset LC(X)$ there exists $f\in\mathcal{O}_{n}$ a  $\mathcal{R}_{X}$-finitely determined function germ such that $df(0)=p$. By Corollary \ref{corteo2},  
$$
\mu_{BR}^{-}(f,X)=\dim_{\C}\frac{\mathcal{O}_{n}}{J(f,\phi)+I_{X}}-\tau(X,0).
$$ 
Let $F:(\C\times\C^{n},0)\to(\C,0)$ be given by $F(t,x)=f_{t}(x)$ a Morsification of $f$. % It follows from \cite[Proposition 5.11]{bruce roberts} that$LC(X)^{-}$ is Cohen-Macaulay at $(0,p)$ if and only if $$\mu_{BR}^{-}(f,X)=\sum_{i=1}^{s+1}m_{i}n_{i},$$ where $n_{i}$ is the number of critical points of $f_{t}$ on $X_{i}$ and $m_{i}$ is the multiplicity of the corresponding irreducible component $Y_{i}$ of $LC(X)^{-}$.
 We consider $R=\mathcal{O}_{n+1}/I_{X}$ and $I=(J(f_{t},\phi)+I_{X})/I_{X}.$ The ring $R$ is Cohen-Macaulay of dimension $n-k+1$ and $I$ is generated by the minors of order $k+1$ of a matrix of size $(k+1)\times n$. Since $\dim R/I=1=\dim R-(n-(k+1)+1)(k+1-(k+1)+1),$ then by Eagon-Hochster, \cite{E-H}, $R/I$ is determinantal and therefore Cohen-Macaulay. By conservation of multiplicity, for all $t\neq 0$, 
$$
\dim_{\C}\frac{\mathcal{O}_{n}}{J(f,\phi)+I_{X}}=\sum_{i=1}^{s+1}\sum_{x\in X_{i}}\dim_{\C}\frac{\mathcal{O}_{n,x}}{J(f_{t},\phi)+I_{X}},
$$
where $X_{i}$ are the logarithmic strata of $(X,0)$ as defined in the beginning of the section. 

When $i=1,...,s$ and $x\in X_{i}$, $X$ is smooth at $x$, so
$$
\sum_{x\in X_{i}}\dim_{\C}\frac{\mathcal{O}_{n,x}}{J(f_{t},\phi)+I_{X}}=\sum_{x\in X_{i}}\mu_{BR}^{-}(f_{t},X)_{x}\stackrel{(*)}=\sum_{x\in X_{i}\cap\Sigma f_{t}}m_{i}=n_{i}m_{i}.
$$
The equality $(*)$ is consequence of \cite[Lemma 4.4]{segundo artigo}. For $i=s+1$, we have just one critical point $x=0$ in $X_{s+1}=\{0\},$ thus
$$
\dim_{\C}\frac{\mathcal{O}_{n}}{J(f_{t},\phi)+I_{X}}=\mu_{BR}^{-}(f_{t},X)+\tau(X,0)=n_{s+1}m_{s+1}+\tau(X,x).
$$
Summing up for all $i=1,...,s+1$ we have 
$$
\dim_{\C}\frac{\mathcal{O}_{n}}{J(f,\phi)+I_{X}}=\sum_{i=1}^{s+1}n_{i}m_{i}+\tau(X,0)
$$ 
and hence $\mu_{BR}^{-}(f,X)=\sum_{i=1}^{s+1}n_{i}m_{i}$. Therefore, $LC(X)^{-}$ is Cohen-Macaulay by \cite[Proposition 5.11]{bruce roberts}
\end{proof}

%Since $LC(X)^{-}$ is Cohen-Macaulay for any ICIS, $(X,0)$, we be able to extend some results of \cite{bruce roberts}.

\begin{cor}
If $(X,0)$ is an ICIS, then $LC(X)$ is Cohen-Macaulay at all points not in $X\times\{0\}$.
\end{cor}
\begin{proof}
Let $(x,p)\in LC(X)$ such that $(x,p)\not\in X\times \{0\}$. If $x\in X$ then $p\neq 0$ and $LC(X)$ coincides, as a complex space, with $LC(X)^-$ on an open neighbourhood of $(x,p)$ in $T^*\C^n$. Hence, $LC(X)$ is Cohen-Macaulay at $(x,p)$ by Theorem \ref{LC(X)menos cm}. 

Otherwise, if $x\notin X$ then $LC(X)$ coincides, as a complex space, with $(\C^{n}\setminus X)\times \{0\}$ on an open neighbourhood of $(x,p)$ in $T^*\C^n$. Again, $LC(X)$ is also Cohen-Macaulay at $(x,p)$.
\end{proof}

%there exist a neighborhood $V$ of $(x,p)$, such that $$LC(X)\cap V=LC(X)^{-}\cap V$$ as a complex space, see the proof of \cite[Lemma 4.4]{segundo artigo}.
%
%
%If $x\not\in X$ then $x\in X_{0}=\C^{n}\setminus X$, that is, $(x,p)\in Y_{0}$ and $I_{Y_{0},(x,p)}$ is a complete intersection. Let $V=\C^{2n}\setminus \cup_{i=1}^{k+1}Y_{i}$ be a open neighborhood of $(x,p)$ and consider the set $$LC(X)\cap V=Y_{0}\cap V.$$ These sets are equal as complex spaces because if $I_{j}$ is the ideal that defines the variety  $Y_{j}$ and $I$ is the ideal that defines $LC(X)$, then $$I=I_{0}\cap...\cap I_{k+1}$$ and since $x\in\C^{n}\setminus X$ then the ideal $I_{j}$ is the total ideal in $V$. Therefore $LC(X)$ is Cohen-Macaulay at this point.

The following corollaries are also motivated by \cite[Proposition 7.4 and Corollary 7.6]{bruce roberts}

\begin{cor}
Let $(X,0)$ be an ICIS and $f:(\C^{n},0)\to(\C,0)$ a function germ. If $f$ has an isolated critical point at $x$, then $$\mu_{BR}(f,X)_{x}\geq\mu_{BR}^{-}(f,X)_{x}+\mu(f)_{x},$$ with equality if either $x\in\C^{n}\setminus X$ or $df(x)\neq 0$. 

Moreover, if $x\in\C^{n}\setminus X$ then $\dim_{\C}\mathcal{O}_{n,x}/df(\Theta_{X,x}^{-})=0$ while if $df(x)\neq 0$, $\mu(f)_{x}=0.$

 If $(X,0)$ is not an isolated hypersurface singularity then the above sufficient conditions for the equality are also necessary.
\end{cor}
\begin{proof}
By \cite[Corollary 5.8 and Propositions {5.11}, 5.14]{bruce roberts} $$\mu_{BR}(f,X)_{x}\geq\sum_{i=0}^{k+1}m_{i}n_{i}=\mu_{BR}^{-}(f,X)_{x}+m_{0}n_{0}=\mu_{BR}^{-}(f,X)_{x}+\mu(f)_{x}.$$
The other statements are consequences of the previous corollary and \cite[Proposition 5.8]{bruce roberts}.
\end{proof}

%\begin{remark}
% Since  the equality $$\mu_{BR}(f,X)_{x}=\mu_{BR}^{-}(f,X)_{x}+\mu(f)_{x},$$ means that $LC(X)$ is Cohen-Macaulay at $(x,df(x))$ we have that $x\not\in X$ or $df(x)\neq 0$. Then when $(X,0)$ is not a hypersurface the above sufficient condition for equality in the previous corollary is also necessary in the previous corollary.
%\end{remark}

%Since a immediate consequence of the previous result we have

\begin{cor}\label{cor usar no ultimo resultado}
Let $(X,0)$ be an ICIS and $f\in\mathcal{O}_{n}$ be a function germ $\mathcal{R}_{X}$-finitely determined. If $x$ is a critical point of $f$ then\begin{align*}
\sum_{i=1}^{k}n_{i}+m_{k+1}&=\mu_{BR}^{-}(f,X)_{x}\\
\sum_{i=0}^{k}n_{i}+m_{k+1}&\leq \mu_{BR}(f,X)_{x}.
\end{align*}
\end{cor}

Let $(X,0)$ be an ICIS determined by $\phi:(\C^{n},0)\to(\C^{k},0)$. Bruce and Roberts define in \cite{bruce roberts} $LC(X)^{T}$ as the complex subspace of $T^{*}\C^{n}$ given by $\phi_{i},\;i=1,...,k$ and   $$I_{k+1}\begin{pmatrix} p_{1}&...&p_{n}\\
\tfrac{\partial\phi_{1}}{\partial x_{1}}&...&\tfrac{\partial\phi_{1}}{\partial x_{n}}\\
\vdots&\ddots&\vdots\\
\tfrac{\partial\phi_{k}}{\partial x_{1}}&...&\tfrac{\partial\phi_{k}}{\partial x_{n}}\\
\end{pmatrix}.$$
 
As observed in \cite{bruce roberts}, the logarithmic stratification obtained by integration of the vectors fields given by the minors 
$$I_{k+1}\begin{pmatrix} 
\tfrac{\partial}{\partial x_{1}}&...&\tfrac{\partial}{\partial x_{n}}\\
\tfrac{\partial\phi_{1}}{\partial x_{1}}&...&\tfrac{\partial\phi_{1}}{\partial x_{n}}\\
\vdots&\ddots&\vdots\\
\tfrac{\partial\phi_{k}}{\partial x_{1}}&...&\tfrac{\partial\phi_{k}}{\partial x_{n}}\\
\end{pmatrix}.$$
is still holonomic. Actually it is the same as that given by $\Theta_{X}$, so $LC(X)^{T}$ is $n$-dimensional with the same irreducible components as $LC(X)^{-}$. Let $Y_{i},\;i=1,...,k$ the irreducible components of $LC(X)^{T}$, then $Y_{i}$ has multiplicity $m_{i}=1$, $i=1,...,k$ and $Y_{k+1}$ has multiplicity denoted by $m(X,0)^{T}$. In general $m(X,0)^{T}$ is greater than the multiplicity of $Y_{k+1}$ in $LC(X)^{-}$, $m_{k+1}$. The principal advantage of considering $LC(X)^{T}$ is that it is Cohen-Macaulay for any ICIS (see \cite[Proposition 7.10]{bruce roberts}). 

\begin{prop}
Let $(X,0)$ be an ICIS and $f\in\mathcal{O}_{n}$ an $\mathcal{R}_{X}$-finitely determined function germ, then $$m(X,0)^{T}-m_{k+1}=\tau(X,0).$$
\end{prop}
\begin{proof}
It is a consequence of Corollary \ref{cor usar no ultimo resultado} and \cite[Corollary 7.11]{bruce roberts}.
\end{proof}

Let $f\in\mathcal{O}_{n}$ be an $\mathcal{R}_{X}$-finitely determined function germ and $F:(\C^{n}\times\C,0)\to(\C,0)$, $F(x,t)=f_{t}(x)$, a 1-parameter deformation of $f$. The polar curve of $F$ with respect to $(X,0)$ is $$C=\{(x,t)\in\C^{n}\times\C;\; df_{t}(\delta_{i})=0\;\forall\; i=1,...,m\},$$ where $\Theta_{X}=\langle\delta_{1},...,\delta_{m}\rangle$. The relative polar curve is the set of points $(x,t)$ in  $C$ such that $x\in X$. %We prove in \cite[Proposition 7.3]{segundo artigo} that if $LC(X)^{-}$ is Cohen-Macaulay then the relative polar curve of every 1-parameter deformation of any $\mathcal{R}_{X}$-finitely determined function germ is Cohen-Macaulay. As a consequence of the previous results we have the principle of conservation of number
As a consequence of Theorem \ref{LC(X)menos cm} (see \cite[Proposition 4.3]{segundo artigo}), $C^{-}$ is Cohen-Macaulay and $$\mu_{BR}^{-}(f,X)=\sum_{x\in\C^{n}}\mu_{BR}^{-}(f_{t},X)_{x}.$$

\section{The Bruce-Roberts number}
 %Let $(X,0)$ be an ICIS with $I_{X}=\langle\phi_{1},...,\phi_{k}\rangle$ and $f$ a $\mathcal{R}_{X}$ finitely determined function germ. 
 In \cite{nunoballesteros oreficeokamoto limapereira tomazela}, we prove that if $(X,0)$ is an IHS  and $f$ an $\mathcal{R}_{X}$- finitely determined then the Bruce-Roberts number of $f$ with respect to $X$ and the Milnor number of $f$ are related by $$\mu_{BR}(f,X)=\mu(f)+\mu(X\cap f^{-1}(0),0)+\mu(X,0)-\tau(X,0).$$

In this section, we study how these invariants are related when $(X,0)$ is an ICIS.

\begin{prop} Let $(X,0)$ be an ICIS and $f\in\mathcal{O}_{n}$ a function germ. Then
 $$\mu_{BR}(f,X)<\infty\textup{ if, and only if, }\dim_{\C}\mathcal{O}_{n}/df(\Theta_{X}^{T})<\infty.$$
\end{prop}  
The proof of this result follows the same idea of Lemma 2.3 in \cite{nunoballesteros oreficeokamoto limapereira tomazela}.

\begin{prop}\label{ICIS dimensao on por dfthetaxt}
Let $(X,0)$ be an ICIS determined by $(\phi_{1},...,\phi_{k}):(\C^{n},0)\to(\C^{k},0)$ and $f\in\mathcal{O}_{n}$ an $\mathcal{R}_{X}$-finitely determined function germ then

$$\mu_{BR}(f,X)=\dim_{\C}\frac{\mathcal{O}_{n}^{k}}{Jf\mathcal{O}_{n}^{k}+\langle\phi_{i}e_{j}-\phi_{j}e_{i}\rangle}+\mu(X\cap f^{-1}(0),0)+\mu(X,0)-\dim_{\C}\frac{df(\Theta_{X})}{df(\Theta_{X}^{T})}.$$

 \end{prop}
\begin{proof}
We consider the following  sequence
$$0\longrightarrow \frac{\mathcal{O}_{n}^{k}}{Jf\mathcal{O}_{n}^{k}+\langle\phi_{i}e_{j}-\phi_{j}e_{i}\rangle}\stackrel{\alpha}{\longrightarrow}\frac{\mathcal{O}_{n}}{df(\Theta_{X}^{T})}\stackrel{\pi}{\longrightarrow}\frac{\mathcal{O}_{n}}{df(\Theta_{X}^{T})+\langle\phi_{1},...,\phi_{k}\rangle}\longrightarrow0$$
 with $\pi$ the projection and $\alpha((a_{1},...,a_{k})+Jf\mathcal{O}_{n}^{k}+\langle\phi_{i}e_{j}-\phi_{j}e_{i}\rangle)=\Sigma_{i=1}^{k}a_{i}\phi_{i}+df(\Theta_{X}^{T}).$
 Obviously, $\pi$ is an epimorphism and $\im\alpha=\ker\pi$. To see the exactness of the sequence it only remains to see that $\alpha$ is a monomorphism.
 
 Since $f$ is $\mathcal{R}_{X}$-finitely determined, $(f,\phi_{1},...,\phi_{k})$ defines an ICIS (\cite[Proposition 2.8]{carlesruas}),
which implies  
 \begin{equation}\label{eq:dim0}
 \dim\frac{\mathcal O_n}{\langle\phi_1,\dots,\phi_k\rangle+J(f,\phi)}=0.
 \end{equation}
This gives $\dim \mathcal O_n/J(f,\phi)\le n-k$. Since $J(f,\phi)$ is the ideal generated by the maximal minors of a matrix of size $n\times (k+1)$, 
 $\mathcal O_n/J(f,\phi)$ is determinantal, and hence, Cohen-Macaulay of dimension $n-k$. Again by \eqref{eq:dim0} we conclude that $\phi_{1}+J(f,\phi),...,\phi_{k}+J(f, \phi)$ is a regular sequence in $\mathcal{O}_{n}/J(f,\phi)$.

We prove now that $\alpha$ is a monomorphism. Let $(a_1,\dots,a_k)\in\mathcal O_n^k$ be such that $\Sigma_{i=1}^{k}a_{i}\phi_{i}\in df(\Theta_{X}^{T})$. By (\ref{dfthetaxt}), $df(\Theta_{X}^{T})=Jf\langle\phi_1,\dots,\phi_k\rangle+J(f,\phi)$, so there exist $\alpha_1,\dots,\alpha_k\in Jf$ such that 
\[
\sum_{i=1}^{k}a_{i}\phi_{i}-\sum_{i=1}^{k}\alpha_{i}\phi_{i}=\sum_{i=1}^{k}(a_{i}-\alpha_i)\phi_{i}\in J(f,\phi).
\]
By the regularity of the classes of $\phi_1,\dots,\phi_k$ in $\mathcal{O}_{n}/J(f,\phi)$,
\[
(a_1-\alpha_1,\dots,a_k-\alpha_k)\in \langle\phi_{i}e_{j}-\phi_{j}e_{i}\rangle+J(f,\phi)\mathcal O_n^k,
\]
which implies  $(a_1,\dots,a_k)\in \langle\phi_{i}e_{j}-\phi_{j}e_{i}\rangle)+J(f)\mathcal O_n^k$.

%\textcolor{red}{From} this fact it is not hard to show that $\alpha$ is a monomorphism. 
 
Finally, the exactness of the sequence implies
  \begin{align*}
 \mu_{BR}(f,X)&=\dim_{\C}\frac{\mathcal{O}_{n}}{df(\Theta_{X}^{T})}-\dim_{\C}\frac{df(\Theta_{X})}{df(\Theta_{X}^{T})}\\
              &=\dim_{\C}\frac{\mathcal{O}_{n}^{k}}{Jf\mathcal{O}_{n}^{k}+\langle\phi_{i}e_{j}-\phi_{j}e_{i}\rangle}+\dim_{\C}\frac{\mathcal{O}_{n}}{df(\Theta_{X}^{T})+\langle\phi_{1},...,\phi_{k}\rangle}-\dim_{\C}\frac{df(\Theta_{X})}{df(\Theta_{X}^{T})}\\
              &=\dim_{\C}\frac{\mathcal{O}_{n}^{k}}{Jf\mathcal{O}_{n}^{k}+\langle\phi_{i}e_{j}-\phi_{j}e_{i}\rangle}+\dim_{\C}\frac{\mathcal{O}_{n}}{J(f,\phi)+\langle\phi_{1},...,\phi_{k}\rangle}-\dim_{\C}\frac{df(\Theta_{X})}{df(\Theta_{X}^{T})}\\
  &=\dim_{\C}\frac{\mathcal{O}_{n}^{k}}{Jf\mathcal{O}_{n}^{k}+\langle\phi_{i}e_{j}-\phi_{j}e_{i}\rangle}+\mu(X\cap f^{-1}(0),0)+\mu(X,0)-\dim_{\C}\frac{df(\Theta_{X})}{df(\Theta_{X}^{T})}, 
  \end{align*}
where the last equality follows from the Lê-Greuel formula.
\end{proof}

%Like we remark before Lemma \ref{caracterizacao do tjurina para usar no relativo}, in previous works in order to relate invariants it was necessary to present a new characterization for the Tjurina number.
 We need another characterization for the Tjurina number which is a generalization for the IHS case obtained in \cite{nunoballesteros oreficeokamoto limapereira tomazela, Tajima}.% for any ICIS $(X,0)$ in terms of the tangent vectors fields of $(X,0)$.

A vector field $\xi$ belongs to $\Theta_{X}$ if and only if there exist a matrix $\left[\lambda_{ij}\right]$ in  the set of the matrices of order $k$ with elements in $\mathcal{O}_{n}$,  $M_{k}(\mathcal{O}_{n})$, such that
\begin{equation}\label{matriz}
d\phi(\xi)=\left[\lambda_{ij}\right]\left[\phi\right]^{T} 
\end{equation}
The matrix $\left[\lambda_{ij}\right]$ is not unique, so we consider the class $\left[\lambda_{ij}\right]$ in $M_{k}(\mathcal{O}_{n})/H$, where $H$ is the submodule generated by the matrices whose rows belong to $\syzygy(\phi_{1},...,\phi_{k})$.

\begin{lem}\label{caracterizacao dos campos triviais usando matrizes}
Let $(X,0)$ be an ICIS determined by $\phi=(\phi_{1},...,\phi_{k}):(\C^{n},0)\to(\C^{k},0)$. Then
 $\xi \in \Theta_{X}^{T}$ if and only if there exist a matrix $[\lambda_{ij}]\in (T+H)/H$ which satisfies (\ref{matriz}), where $T$ is the submodule generated by the matrices $T_{lm},\;l=1,...,k;\;m=1,...,n$, such that $l$-th column is equal to $m$-th column of the jacobian matrix of $\phi$ and the other columns are null.  
\end{lem}
\begin{proof}
 Let $\xi\in\Theta_{n}$ and $[\lambda_{ij}]+H\in (T+H)/(H),$ such that (\ref{matriz}) holds, so there exist  $\alpha_{ij}\in\mathcal{O}_{n}$ and matrices $[h_{lm}]\in H$ such that $$[\lambda_{ij}]=\sum_{j=1}^{n}\sum_{i=1}^{p}\alpha_{ij}T_{ij}+[h_{lm}].$$ 
Therefore,
 $$d\phi(\xi)=[\lambda_{ij}][\phi]^{T}=\left(\sum_{j=1}^{n}\sum_{i=1}^{p}\alpha_{ij}T_{ij}+[h_{lm}]\right)[\phi]^{T}=\left(\sum_{j=1}^{n}\sum_{i=1}^{p}\alpha_{ij}T_{ij}\right)[\phi]^{T}=d\phi(\eta),$$
with $\eta=(\sum_{i=1}^{k}\alpha_{l1}\phi_{l},...,\sum_{i=1}^{k}\alpha_{ln}\phi_{l})\in\Theta_{X}^{T},$
and $\xi\in\Theta_{X}^{T}$.

The converse is immediate.
\end{proof}

When $(X,0)$ is an IHS and $f$ is an $\mathcal{R}_{X}$-finitely determined function germ, any $\xi \in\Theta _{X}$ such that $df(\xi)\in I_{X}$ is trivial. We will use this result for any ICIS, in order to prove it we need the following lemma.

\begin{lem}\label{icis e sequencia regular no modulo quociente im}
Let $(X,0)$ an ICIS determined by $\phi=(\phi_{1},...,\phi_{k}):(\C^{n},0)\to(\C^{k},0).$ Then
$\phi_{1},...,\phi_{k-1}$ is a regular sequence in $\mathcal{O}_{n}^{k}/ \im (d\phi)$.
\end{lem}
\begin{proof} 
 We denote  $M=\mathcal{O}_{n}^{k}/\im(d\phi)$. The sequence 
$$\mathcal{O}_{n}^{n}\stackrel{d\phi}\longrightarrow\mathcal{O}_{n}^{k}\longrightarrow M\longrightarrow 0,$$ is a presentation of $M$, the $0$-th fitting ideal of $M$ is $F_{0}(M)=J(\phi_{1},...,\phi_{k})$ and 
$$\dim(M)=\dim\frac{\mathcal{O}_{n}}{\Ann(M)}=\dim\frac{\mathcal{O}_{n}}{J(\phi_{1},...,\phi_{k})}=k-1.$$
Hence, $M$ is a $\mathcal{O}_{n}$-module Cohen-Macaulay by \cite{buchsbaum rim}.

The module $M_{X}=\mathcal{O}_{X}^{k}/\im d\phi\approx \mathcal{O}_{n}^{k+1}/(\im(d\phi)+\langle\phi_{1},...,\phi_{k}\rangle\mathcal{O}_{n}^{k})$ has the following presentation 
$$\mathcal{O}_{X}^{n}\stackrel{d\phi}\longrightarrow\mathcal{O}_{X}^{k}\longrightarrow M_{X}\longrightarrow 0,$$
and the $0$-th fitting ideal is $J(\phi_{1},...,\phi_{k})$. Thus,
$$\dim(M_{X})=\dim\frac{\mathcal{O}_{X}}{\Ann(M_{X})}=\dim\frac{\mathcal{O}_{n}}{\langle\phi_{1},...,\phi_{k-1}\rangle+J(\phi_{1},...,\phi_{k})}=0,$$ and therefore $\phi_{1},...,\phi_{k-1}$ is a regular sequence in $\mathcal{O}_{n}^{k}/ \im (d\phi)$.
\end{proof} 
%Let $(X,0)$ be an ICIS determined by $(\phi_{1},...,\phi_{k}):(\C^{n},0)\to(\C,0)$ and $f\in\mathcal{O}_{n}$, if $f$ is a $\mathcal{R}_{X}$-finitely determined function germ then $(\phi_{1},...,\phi_{k},f):(\C^{n},0)\to(\C^{k+1})$ defines an ICIS. We know that the converse is true with the additional hypothesis $f$ has isolated singularity. The next result gives information about the vectors fields $\xi\in\Theta_{X}$ with $df(\xi)\in I_{X}$, $f$ an $\mathcal{R}_{X}$-finitely determined function germ.
\begin{teo}\label{ICIS se df do campo pertence a I(x) o campo e trivial }\label{icis independe de f}
Let $(X,0)$ be an ICIS determined by $(\phi_{1},...,\phi_{k}):(\C^{n},0)\to(\C^{k},0)$ and $f\in\mathcal{O}_{n}$ an $\mathcal{R}_{X}$-finitely determined function germ.
If $\xi\in\Theta_{X}$ and $df(\xi)\in I_{X}$, then $\xi\in\Theta_{X}^{T}.$
In particular the evaluation map $E:\Theta_{X}\to df(\Theta_{X})$ given by $E(\xi)=df(\xi)$ induces an isomorphism $$\overline{E}:\frac{\Theta_{X}}{\Theta_{X}^{T}}\to \frac{df(\Theta_{X})}{df(\Theta_{X}^{T})}.$$

\end{teo}

\begin{proof}

Since $\xi\in\Theta_{X}$ and $df(\xi)\in I_{X}$, there exist $\lambda_{ij},\mu_j\in\mathcal{O}_{n}$ with $i,j=1,...,k$ such that $d\phi_i(\xi)=\Sigma_{j=1}^k \lambda_{ij}\phi_j$ and
$df(\xi)=\Sigma_{j=1}^k \mu_j\phi_j$. This gives
$$d(f,\phi)(\xi)= \sum_{j=1}^{k}\phi_{j}\begin{pmatrix}
\mu_{j}\\
\lambda_{1j}\\
\vdots\\
\lambda_{kj}
\end{pmatrix}=\sum_{j=1}^{k}\phi_{j}v_{j}^{T},$$
where $v_{j}^{T}$ is the transpose of $v_{j}=(\mu_{j},\lambda_{1j},...,\lambda_{kj}).$
Therefore, $\sum_{j=1}^{k}\phi_{j}v_{j}^{T}\in \im d(f,\phi)$ and 
\[
\phi_{k}v_{k}\in \im d(f,\phi)+\langle\phi_{1},...,\phi_{k-1}\rangle\mathcal{O}_{n}^{k+1}.
\]

However, since $f$ is $\mathcal{R}_{X}$-finitely determined, $(f,\phi)$ defines an ICIS. By Lemma \ref{icis e sequencia regular no modulo quociente im}, $\phi_{1},...,\phi_{k}$ is a regular sequence in  $\mathcal{O}_{n}^{k+1}/\im d(f,\phi)$, so 
\[
v_k\in \im d(f,\phi)+\langle\phi_{1},...,\phi_{k-1}\rangle\mathcal{O}_{n}^{k+1}.
\]  
That is, there exist $w_k\in\im d(f,\phi)$ and $a^{k}_{ij}\in\mathcal{O}_{n}$, such that
$$
v_{k}=w_{k}+\sum_{i=1}^{k-1}\sum_{j=1}^{k+1}a^{k}_{ij}\phi_{i}e_{j}.
$$
Therefore, 
\begin{align*}
\sum_{j=1}^k\phi_{i}v_{i}&=\phi_{1}v_{1}+...+\phi_{k}\left(w_{k}+\sum_{i=1}^{k-1}\sum_{j=1}^{k+1}a^{k}_{ij}\phi_{i}e_{j}\right)\\
&=\phi_{1}\left(v_{1}+\sum_{j=1}^{k+1}a_{1j}^{k}\phi_{k}e_{j}\right)+...+\phi_{k-1}\left(v_{k-1}+\sum_{j=1}^{k+1}a_{(k-1)j}^{k}\phi_{k}e_{j}\right)+\phi_kw_k.
\end{align*} 
Hence,
\[
\phi_{k-1}\left(v_{k-1}+\sum_{j=1}^{k+1}a_{(k-1)j}^{k}\phi_{k}e_{j}\right)\in \im d(f,\phi)+\langle\phi_{1},...,\phi_{k-2}\rangle\mathcal{O}_{n}^{k+1},
\] 
and thus,  
\[
v_{k-1}+\sum_{j=1}^{k+1}a_{(k-1)j}^{k}\phi_{k}e_{j}\in\im d(f,\phi)+\langle\phi_{1},...,\phi_{k-2}\rangle\mathcal{O}_{n}^{k+1}.
\] 
That is, there exist $w_{k-1}\in\im d(f,\phi)$ and $a^{k-1}_{ij}\in\mathcal{O}_{n}$ such that $$v_{k-1}+\sum_{j=1}^{k+1}a_{(k-1)j}^{k}\phi_{k}e_{j}=w_{k-1}+\sum_{i=1}^{k-2}\sum_{j=1}^{k+1}a_{ij}^{k-1}\phi_{i}e_{j}.$$

We have 
\begin{align*}
&\phi_{1}\left(v_{1}+\sum_{j=1}^{k+1}a_{1j}^{k}\phi_{k}e_{j}\right)+...+\phi_{k-1}\left(v_{k-1}+\sum_{j=1}^{k+1}a_{(k-1)j}^{k}\phi_{k}e_{j}\right)+\phi_kw_k=\\
&\phi_{1}\left(v_{1}+\sum_{j=1}^{k+1}a_{1j}^{k}\phi_{k}e_{j}\right)+...+\phi_{k-1}\left(w_{k-1}+\sum_{i=1}^{k-2}\sum_{j=1}^{k+1}a_{ij}^{k-1}\phi_{i}e_{j}\right)+\phi_kw_k=\\
&\phi_{1}\left(v_{1}+\sum_{j=1}^{k+1}a_{1j}^{k}\phi_{k}e_{j}+\sum_{j=1}^{k+1}a_{1j}^{k-1}\phi_{k-1}e_{j}\right)+...+\\
&\phi_{k-2}\left( v_{k-2}+\sum_{j=1}^{k+1}a_{(k-2)j}^{k}\phi_{k}e_{j}+\sum_{j=1}^{k+1}a_{(k-2)j}^{k-1}\phi_{k-1}e_{j}\right)+\phi_{k-1}w_{k-1}+\phi_kw_k
\end{align*}
and again 
\[
\phi_{k-2}\left( v_{k-2}+\sum_{j=1}^{k+1}a_{(k-2)j}^{k}\phi_{k}e_{j}+\sum_{j=1}^{k+1}a_{(k-2)j}^{k-1}\phi_{k-1}e_{j}\right)\in\im d(f,\phi)+\langle\phi_{1},...,\phi_{k-3}\rangle\mathcal{O}_{n}^{k+1}.
\]

 Proceeding like this, we show that $$\phi_{1}\left(v_{1}+\sum_{j=1}^{k+1}a_{1j}^{k}\phi_{k}e_{j}+...+\sum_{j=1}^{k+1}a_{1j}^{3}\phi_{3}e_{j}\right)+\phi_{2}\left(v_{2}+\sum_{j=1}^{k+1}a_{2j}^{k}\phi_{k}e_{j}+...+\sum_{j=1}^{k+1}a_{2j}^{3}\phi_{3}e_{j}\right)$$
is in  $\im d(f,\phi)$. This implies
\[
\phi_{2}\left(v_{2}+\sum_{j=1}^{k+1}a_{2j}^{k}\phi_{k}e_{j}+...+\sum_{j=1}^{k+1}a_{2j}^{3}\phi_{3}e_{j}\right)\in\im d(f,\phi)+\langle\phi_{1}\rangle\mathcal{O}_{n}^{k+1}.
\] 
That is, there exist $w_{2}\in\im d(f,\phi_{1},...,\phi_{k})$ and $a_{1j}^{2}\in\mathcal{O}_{n}$ such that 
$$v_{2}+\sum_{j=1}^{k+1}a_{2j}^{k}\phi_{k}e_{j}+...+\sum_{j=1}^{k+1}a_{2j}^{3}\phi_{3}e_{j}=w_{2}+\sum_{j=1}^{k+1}a_{1j}^{2}\phi_{1}e_{j}.$$ %logo $$v_{2}=w_{2}+\sum_{j=1}^{k+1}a_{1j}^{2}\phi_{1}e_{j}-\sum_{j=1}^{k+1}a_{2j}^{k}\phi_{k}e_{j}-...-\sum_{j=1}^{k+1}a_{2j}^{3}\phi_{3}e_{j},$$  
We arrive to
\begin{align*}
&\phi_{1}\left(v_{1}+\sum_{j=1}^{k+1}a_{1j}^{k}\phi_{k}e_{j}+...+\sum_{j=1}^{k+1}a_{1j}^{3}\phi_{3}e_{j}\right)+\phi_{2}\left(v_{2}+\sum_{j=1}^{k+1}a_{2j}^{k}\phi_{k}e_{j}+...+\sum_{j=1}^{k+1}a_{2j}^{3}\phi_{3}e_{j}\right)=\\
&\phi_{1}\left(v_{1}+\sum_{j=1}^{k+1}a_{1j}^{k}\phi_{k}e_{j}+...+\sum_{j=1}^{k+1}a_{1j}^{3}\phi_{3}e_{j}\right)+\phi_{2}\left(w_2+\sum_{j=1}^{k+1}a_{1j}^{2}\phi_{1}e_{j}\right)\in\im d(f,\phi).\\
\end{align*}
Therefore, 
\[
\phi_{1}\left(v_{1}+\sum_{j=1}^{k+1}a_{1j}^{k}\phi_{k}e_{j}+...+\sum_{j=1}^{k+1}a_{1j}^{3}\phi_{3}e_{j}+\sum_{j=1}^{k+1}a_{1j}^{2}\phi_{2}e_{j}\right)\in\im d(f,\phi),
\] 
and there exist $w_{1}\in\im d(f,\phi)$ such that $$v_{1}+\sum_{j=1}^{k+1}a_{1j}^{k}\phi_{k}e_{j}+...+\sum_{j=1}^{k+1}a_{1j}^{3}\phi_{3}e_{j}+\sum_{j=1}^{k+1}a_{ij}^{2}\phi_{2}e_{j}=w_{1}.$$
We conclude that there are $w_{1},...,w_{k}\in\im d(f,\phi)$ such that   
\begin{align*}
v_{1}&=w_{1}-\sum_{j=1}^{k+1}a_{1j}^{k}\phi_{k}e_{j}-...-\sum_{j=1}^{k+1}a_{1j}^{3}\phi_{3}e_{j}-\sum_{j=1}^{k+1}a_{ij}^{2}\phi_{2}e_{j}\\
v_{2}&=w_{2}+\sum_{j=1}^{k+1}a_{1j}^{2}\phi_{1}e_{j}-\sum_{j=1}^{k+1}a_{2j}^{k}\phi_{k}e_{j}-...-\sum_{j=1}^{k+1}a_{2j}^{3}\phi_{3}e_{j}\\
&\vdots\\
v_{k-1}&=w_{k-1}+\sum_{j=1}^{k+1}a_{(k-2)j}^{k-1}\phi_{k-2}e_j+...+\sum_{j=1}^{k+1}a_{1j}^{k-1}\phi_{1}e_{j}-\sum_{j=1}^{k+1}a_{(k-1)j}^{k}\phi_{k}e_{j}\\
v_{k}&=w_{k}+\sum_{j=1}^{k+1}a_{(k-1)j}^{k}\phi_{k-1}e_{j}+...+\sum_{j=1}^{k+1}a_{1j}^{k}\phi_{1}e_{j}.
\end{align*}
Thus,%$\phi_{1}v_{1}+...+\phi_{k}v_{k}=\phi_{1}w_{1}+...\phi_{k}w_{k}$, and 
$$\phi_{1}(v_{1}-w_{1})+...+\phi_{k}(v_{k}-w_{k})=0.$$

Since $v_{i}^{T}=(\mu_{i},\lambda_{1i},...,\lambda_{ki})^{T}$ and $w_{i}=(\alpha_{i},\alpha_{1i},...,\alpha_{ki})^{T}\in\im d(f,\phi)$ for each $i=1,...,k$,  
we have $[\lambda_{ij}]-[\alpha_{ij}]\in H$, and consequently $\xi\in\Theta_{X}^{T}$, by Lemma \ref{caracterizacao dos campos triviais usando matrizes}.  

\end{proof}

The previous Theorem shows that the quotient $df(\Theta_{X})/df(\Theta_{X}^{T})$ does not depend on the $\mathcal{R}_{X}$-finitely determined function germ $f$. In fact, we show in the next proposition that its dimension as a $\C$-vector space is equal to the Tjurina number of $(X,0)$. This extends the IHS case considered in \cite{nunoballesteros oreficeokamoto limapereira tomazela, Tajima}.

%Now we are ready to characterize the Tjurina number in terms of the tangent vectors fields of an ICIS. 

\begin{prop}\label{tjurina para o bruce roberts}
Let $(X,0)$ be an ICIS and $f\in\mathcal{O}_{n}$ an $\mathcal{R}_{X}$-finitely determined function germ, then $$\dim_{\C}\frac{\Theta_{X}}{\Theta_{X}^{T}}=\dim_{\C}\frac{df(\Theta_{X})}{df(\Theta_{X}^{T})}=\tau(X,0).$$

%Let $p\in\mathcal{O}_{n}$ be a generic linear projection $\mathcal{R}_{X}$-finitely determined then $$\dim_{\C}\frac{dp(\Theta_{X})}{dp(\Theta_{X}^{T})}=\tau(X,0).$$
\end{prop}
\begin{proof}
As the dimension $\dim_{\C}df(\Theta_{X})/df(\Theta_{X}^{T})$ does not depend of $f$ we consider $p\in\mathcal{O}_{n}$ a generic linear projection and the following exact sequence
$$0\longrightarrow\ker(\overline{\alpha})\stackrel{i}\longrightarrow\frac{\mathcal{O}_{n}}{dp(\Theta_{X}^{T})}\stackrel{\overline{\alpha}}\longrightarrow\frac{\mathcal{O}_{n}^{k+1}}{\im (d(p,\phi))+I_X\mathcal{O}_{n}^{k}}\stackrel{\overline{\pi}}\longrightarrow\frac{\mathcal{O}_{n}^{k}}{\im(d(\phi))+I_X\mathcal{O}_{n}^{k}}\longrightarrow 0,$$
where $i$ is the inclusion, $\overline{\pi}$ is induced by the projection $\pi:\mathcal{O}_{n}^{k+1}\to\mathcal{O}_{n}^{k}$ given by $\pi(a_{0},...,a_{k})=(a_{1},...,a_{k})$ and $\overline{\alpha}$ is the map induced by $\alpha:\mathcal{O}_{n}\to\mathcal{O}_{n}^{k+1}$ given by $\alpha(a)=(a,0,...,0).$

Since $p$ is a generic projection, $dp(\Theta_{X})=dp(\Theta_{X})+I_{X}$ and $dp(\Theta_{X}^{T})=dp(\Theta_{X}^{T})+I_{X}$. Now we proceed as in the proof of Theorem \ref{relation}.
\end{proof}

In the proof of Theorem \ref{ICIS se df do campo pertence a I(x) o campo e trivial } we observe that
$$\frac{df(\Theta_{X}^{-})}{df(\Theta_{X})}\approx\frac{\mathcal{O}_{n}^{k}}{Jf\mathcal{O}_{n}^{k}+\syzygy(\phi_{1},...,\phi_{k})}\approx\frac{I_{X}}{JfI_{X}}.$$

%\begin{cor}\label{I/JFI}Let $(X,0)$ be an ICIS determined by $(\phi_{1},...,\phi_{k}):(\C^{n},0)\to(\C^{k},0)$ and $f\in\mathcal{O}_{n}$ a $\mathcal{R}_{X}$-finitely determined function germ.
%\begin{enumerate}
%\item[i)] If exist $\xi \in \Theta_{X}$, such that $df(\xi)=\sum_{i=1}^{k}\mu_{i}\phi_{i}$, then $(\mu_{1},...,\mu_{k})\in Jf\mathcal{O}_{n}^{k}+\syzygy(\phi_{1},...,\phi_{k})$.
%\begin{multicols}{2}
%\item[ii)] $\tfrac{df(\Theta_{X}^{-})}{df(\Theta_{X})}\approx\tfrac{\mathcal{O}_{n}^{k}}{Jf\mathcal{O}_{n}^{k}+\syzygy(\phi_{1},...,\phi_{k})};$
%\item[iii)]$\tfrac{df(\Theta_{X}^{-})}{df(\Theta_{X})}\approx\tfrac{I_{X}}{JfI_{X}}.$
%\end{multicols}
%\end{enumerate}
%\end{cor}

%Hence by Theorem \ref{icis independe de f}, $\dim_{\C}df(\Theta_{X})/df(\Theta_{X}^{T})$ does not depend on the homolomorphic function germ, we conclude from Proposition \ref{tjurina para o bruce roberts} that  
% $$\tau(X,0)=\dim_{\C}\Theta_{X}/\Theta_{X}^{T}=\dim_{\C}df(\Theta_{X})/df(\Theta_{X}^{T})$$
Therefore from Propositions \ref{ICIS dimensao on por dfthetaxt} and \ref{tjurina para o bruce roberts} we have %and Corollary \ref{I/JFI}
 \begin{equation}\label{igualdade que vou usar em baixo}
 \mu_{BR}(f,X)=\dim_{\C}I_{X}/JfI_{X}+\mu(f^{-1}(0)\cap X,0)+\mu(X,0)-\tau(X,0).
 \end{equation}

\begin{teo}\label{relacao usando tor}
Let $(X,0)$ be an ICIS and $f\in\mathcal{O}_{n}$ an $\mathcal{R}_{X}$-finitely determined germ, then 
\begin{align*}
\mu_{BR}(f,X)&=\mu(f)+\mu(X\cap f^{-1}(0),0)+\mu(X,0)-\tau(X,0)-\dim_{\C}\frac{\mathcal{O}_{n}}{Jf+I_{X}}&\\
&\hspace{9.0cm}+\dim_{\C}\frac{I_{X}\cap Jf}{I_{X}Jf}.%\dim_{\C}\Tor_{1}^{\mathcal{O}_{n}}\left(\frac{\mathcal{O}_{n}}{I_{X}},\frac{\mathcal{O}_{n}}{Jf}\right)
\end{align*}
\end{teo}
\begin{proof} We consider the following exact sequence,
$$0\longrightarrow \frac{I_{X}\cap Jf}{I_{X}Jf}\stackrel{i}\longrightarrow \frac{I_{X}}{I_{X}Jf} \stackrel{\pi}\longrightarrow \frac{I_{X}}{I_{X}\cap Jf}\longrightarrow 0, $$
hence, \begin{align*}
\dim_{\C}\frac{I_{X}}{I_{X}Jf}&=\dim_{\C}\frac{I_{X}\cap Jf}{I_{X}Jf}+\dim_{\C}\frac{I_{X}}{I_{X}\cap Jf}\\
                              &=\dim_{\C}\frac{I_{X}\cap Jf}{I_{X}Jf}+\dim_{\C}\frac{I_{X}+Jf}{Jf}.
                             %&=\dim_{\C}\Tor^{\mathcal{O}_{n}}_{1}\left(\frac{\mathcal{O}_{n}}{I_{X}},\frac{\mathcal{O}_{n}}{Jf}\right)+\mu(f)-\dim_{\C}\frac{\mathcal{O}_{n}}{I_{X}+Jf}.
                             \end{align*}
Using the equality (\ref{igualdade que vou usar em baixo}) we conclude the proof.

%Where $(*)$ is consequence of the isomorphisms $$\Tor^{\mathcal{O}_{n}}_{1}\left(\frac{\mathcal{O}_{n}}{I_{X}},\frac{\mathcal{O}_{n}}{Jf}\right)\approx \frac{I_{X}\cap Jf}{JfI_{X}}\textup{ and } \frac{I_{X}}{I_{X}\cap Jf}\approx \frac{I_{X}+Jf}{Jf}.$$ 
\end{proof}

In general, to calculate the dimension $\dim_{\C}(I_{X}\cap Jf)/(I_{X}Jf)$ is not easy. In order to improve the formula for $\mu_{BR}(f,X)$ in Theorem \ref{relacao usando tor}, we observe  
\[
\frac{I_{X}\cap Jf}{I_{X}Jf} \approx\Tor_{1}^{\mathcal{O}_{n}}\left(\frac{\mathcal{O}_{n}}{I_{X}},\frac{\mathcal{O}_{n}}{Jf}\right), 
\]
(see \cite{greuel pfister}). 

Comparing the formula for $\mu_{BR}(f,X)$ in the previous theorem with the formula of \cite{nunoballesteros oreficeokamoto limapereira tomazela} in the IHS case, we get the following:

\begin{cor}\label{hipersuperfice}
Let $(X,0)\subset(\C^{n},0)$ be an isolated hypersurface singularity % such that $I_{X}=\langle\phi\rangle$ 
 and $f$ an $\mathcal{R}_{X}$-finitely determined function germ. Then,
$$\dim_{\C}\Tor_{1}^{\mathcal{O}_{n}}\left(\frac{\mathcal{O}_{n}}{I_{X}},\frac{\mathcal{O}_{n}}{Jf}\right)=\frac{\mathcal{O}_{n}}{I_{X}+Jf}.$$
\end{cor}
 
In order to improve our formula we conjecture:

\begin{conj}\label{conj}  Let $R$ be a regular local ring of dimension $n$, $I$ an ideal in $R$ generated by a regular sequence of length $k\le n$ and $J$ an ideal in $R$ generated by a regular sequence of length $n$. Then,
$$\length\left(\Tor_{i}^{R}\left(\frac{R}{I},\frac{R}{J}\right)\right)=\begin{pmatrix}

k\\
i
\end{pmatrix}\length\left(\frac{R}{J+I}\right).$$
\end{conj}

We observe that the conjecture is known to be true in the particular case that $J\subset I$, see \cite[Lemma 3.4.1]{teseborna}.
 
\medskip
If the conjecture is true, then for any ICIS $(X,0)$ of codimension $k$ and any $\mathcal{R}_{X}$-finitely determined function germ $f$,
$$\mu_{BR}(f,X)=\mu(f)+\mu(X\cap f^{-1}(0),0)+\mu(X,0)-\tau(X,0)+(k-1)\dim_{\C}\frac{\mathcal{O}_{n}}{Jf+I_{X}}.$$

\subsection{ICIS of codimension 2}
%In this subsection we consider $(X,0)$ an ICIS of codimension 2 such that $I_{X}=\langle\phi_{1},\phi_{2}\rangle$ and  $f\in\mathcal{O}_{n}$ an $\mathcal{R}_{X}$-finitely determined function germ. 

The following proposition gives a proof of Conjecture \ref{conj} when $R=\mathcal O_n$ and $k=2$.

%\textcolor{blue}{Nesta proposição não precisamos um ICIS de codimensão 2, basta que seja interseção completa. Eu simplifiquei a prova um pouco, no precisa usar o ideal $I:J$, podemos usar direitamente o anulador de $\bar I$.}

\begin{prop}\label{dimension of tor1}
Let $I\subset \mathcal{O}_{n}$ be an ideal generated by a regular sequence of length $2$ and 
$J\subset \mathcal{O}_{n}$ an ideal generated by a regular sequence of length $n$, then $$\dim_{\C}\Tor_{1}^{\mathcal{O}_{n}}\left(\frac{\mathcal{O}_{n}}{I},\frac{\mathcal{O}_{n}}{J}\right)=2\dim_{\C}\frac{\mathcal{O}_{n}}{J+I}$$
\end{prop}

\begin{proof}
Assume $I$ is generated by the regular sequence $\phi_1,\phi_2$ and consider the Koszul complex, 
$$0\longrightarrow \mathcal{O}_{n}\stackrel{\psi_{2}}\longrightarrow\mathcal{O}_{n}^{2}\stackrel{\psi_{1}}\longrightarrow\mathcal{O}_{n}\longrightarrow 0,$$ where $\psi_{2}(\alpha)=\alpha(\phi_{2},-\phi_{1})$ and $\psi_{1}(\alpha,\beta)=\alpha\phi_{1}+\beta\phi_{2}$. 
Tensoring with $R:=\mathcal{O}_{n}/J$, we obtain $$ 0\longrightarrow R\stackrel{\Psi_{2}}\longrightarrow R^2\stackrel{\Psi_{1}}\longrightarrow R\longrightarrow 0,$$ and $\Tor_{1}^{\mathcal{O}_{n}}(\mathcal{O}_{n}/I,R)=\ker(\Psi_{1})/\im(\Psi_{2}),$ where $\Psi_{1}$ and $\Psi_{2}$ are induced maps by $\psi_{1}$ and $\psi_{2}$, respectively.

The image of $\Psi_{1}$ is equal to  $\bar I:=(I+J)/J$, hence
$$
\dim_{\C}\ker\Psi_{1}=2\dim_{\C}R-\dim_{\C}\bar I=\dim_{\C}R+\dim_{\C}\frac{R}{\bar I}.
$$
The kernel of $\Psi_2$ is $\Ann(\bar I)$, the annihilator of $\bar{I}$ in $R$, so
$$
\dim_{\C}\im\Psi_{2}=\dim_{\C}R-\dim_{\C}\Ann(\bar I)=\dim_\C\frac{R}{\Ann(\bar I)}.
$$

By \cite[Proposition 11.4]{nuno mond}, there exists a perfect pairing on $R$, that is, a symmetric non degenerate bilinear form $$\langle\;\cdot\; ,\;\cdot\;\rangle:R\times R\rightarrow \C$$ such that $\sigma:R\rightarrow R^{*}$ defined by $\sigma(a)=\langle\;\;, a\rangle$ is an isomorphism.
 
 Let $g_{1},...,g_{\mu}$ be a basis over $\C$ of $R$ such that $g_{1},...,g_{r}$ is basis of $\overline{I}$.
 By using the dual basis in $R^{*}$ and the isomorphism $\sigma$ we get a basis of $R$, $h_{1},...,h_{\mu}$, such that $$\langle g_{i},h_{j}\rangle=\delta_{ij}.$$
 
 Let $I^{\perp}=\{a\in R;\;\langle a, b\rangle=0,\;\forall b\in\overline{I}\}$. Then $I^{\perp}$ is generated over $\C$ by $h_{r+1},...,h_{\mu}.$ By \cite[Proposition 3.2(i)]{eisenbud levine}, $I^{\perp}=\Ann(\overline{I})$.
 
 Therefore $\dim_{\C}\Ann(\overline{I})=\mu-r$ and $$\dim_{\C}\im\Psi_{2}=\dim_{\C}\frac{R}{\Ann(\overline{I})}=r=\dim_{\C}\overline{I}.$$

To conclude, \begin{align*}
\dim_{\C}\Tor_{1}^{\mathcal{O}_{n}}\left(\frac{\mathcal{O}_{n}}{I},\frac{\mathcal{O}_{n}}{J}\right)&=\dim_{\C}\ker\Psi_{2}-\dim_{\C}\im\Psi_{1}\\
&=\dim_{\C}R+\dim_{\C}\frac{R}{\bar I}-\dim_{\C} \overline{I}\\
&=2\dim_{\C}\frac{R}{\bar I}\\
&=2\dim_{\C}\frac{\mathcal{O}_{n}}{J+I}.
\end{align*}  
\end{proof}

\begin{cor}\label{caso k=2}Let $(X,0)$ be an ICIS of codimension $2$ and $f$ an $\mathcal{R}_{X}$-finitely determined function germ, then
$$\mu_{BR}(f,X)=\mu(f)+\mu(f^{-1}(0)\cap X,0)+\mu(X,0)-\tau(X,0)+\dim_{\C}\frac{\mathcal{O}_{n}}{Jf+\langle\phi_{1},\phi_{2}\rangle}.$$ 
\end{cor}

\printindex

\begin{thebibliography}{20}
%\markboth{ReferÃªncias BibliogrÃ¡ficas}{ReferÃªncias BibliogrÃ¡ficas}  
\addcontentsline{toc}{chapter}{ReferÃªncias BibliogrÃ¡ficas}
%
\bibitem{tomazellaruas}
  I. Ahmed, M. A. S. Ruas, J. N. Tomazella,\
 {\em Invariants of topological relative right equivalences}, Mathematical Proceedings of the Cambridge Philosophical Society, 155 (2013), no. 2, 307--315.
%      
%    
%
\bibitem{carlesruas}
  C. Bivi\`a-Ausina, M. A. S. Ruas,\
  {\em Mixed Bruce-Roberts number},  Proc. Edinb. Math. Soc. (2),  63, (2020), no. 2, 456--474.

\bibitem{teseborna}
  K. Borna,\
{\em Betti numbers of modules over Noetherian rings with applications to local cohomology}, Thesis, Tehran, Iran, 2008.

%\bibitem{greuel}
%  E. Brieskorn and G. M. Greuel,\
%    \emph{Singularities of complete intersections}, Manifolds-Tokyo 1973 (Proc. Internat. Conf., Tokyo, 1973), Univ. Tokyo Press, (1975), 123--129.

\bibitem{bruce roberts}
  J. W. Bruce, R. M. Roberts,\
  {\em Critical points of functions on analytic varieties}, Topology {27} (1988), No. 1, 57--90.

\bibitem{buchsbaum rim}
  D. A. Buchsbaum, D. S. Rim,\
{\em A generalized Koszul complex. II. Depth and multiplicity}.Trans. Amer. Math. Soc.{111} (1964), 197--224.

\bibitem{singular}
W. Decker, G. M. Greuel, G. Pfister, H. Sch{\"o}nemann,  
\newblock {\sc Singular} {4-3-0} -- {A} computer algebra system for polynomial computations.
\newblock {https://www.singular.uni-kl.de} (2022).

%
\bibitem{E-H}
J. A. Eagon, M. Hochster,\
{\em  Cohen--Macaulay rings, invariant theory, and the generic perfection
of determinantal loci}, Amer. J. Math. { 93} (1971), 1020--1058.

\bibitem{eisenbud levine}
D. Eisenbud, H. I. Levine,\
{\em An algebraic formula for the degree of a $\C^{\infty}$ map germ}, Ann. of Math. (2) 106, 1977, no. 1, 19--44. 


\bibitem{Gaffney}
  T. Gaffney,\
  {\em Multiplicities an equisingularity of ICIS germs}, Invent. Math. 123 (1996), No. 2, 209-220.
%  
%\bibitem{greuel2} G. M. Greuel, {\it Constant Milnor number implies constant multiplicity for quasihomogeneous singularities}, Manuscripta Math.
%{56} (1986), No. 2, 159-166.

\bibitem{greuel pfister} G. M. Greuel, G. Pfister,\
{\em A Singular introduction to commutative algebra.} Second, extended edition. With contributions by Olaf Bachmann, Christoph Lossen and Hans Sch\"onemann. Springer, Berlin, 2008.

\bibitem{greuel deformations} 
 G. M. Greuel, C. Lossen, E. Shustin,\
 {\em Introduction to singularities and deformations.} Springer Monographs in Mathematics. Springer, Berlin, 2007.

 


%\bibitem{Nivaldoerrata}
%N. G. Grulha JÃºnior,\ 
%{\em Erratum: The Euler obstruction and Bruce-Roberts' Milnor number} Quarterly Journal of Mathematics, v. 63, 257--258, 2012.

\bibitem{Nivaldo}
  N. G. Grulha J\'unior,\
{\em The Euler Obstruction and Bruce--Roberts' Milnor Number}, Quarterly Journal of Mathematics, v. 60 (2009), 291--302. 
%%
%\bibitem{hamm2}
%  H. Hamm,\
%  \emph{Lokale topologische Eigenschaften komplexer R$\ddot{a}$ume}, (German) Math. Ann. { 191} (1971) 235-252.

%\bibitem{hamm1}
%  H. Hamm,\
%    \emph{Topology of isolated singularities of complex spaces}, Proceedings of Liverpool Singularities Symposium, II (1969/1970), 213--217. Lectures Noted in Math., { 209}, Springer, Berlin 1971
%
%\bibitem{hungerford}
%T. W. Hungerford; {\em Algebra}, volume 73 of Graduate Texts in Mathematics. Springer-Verlag, New York, 1980. Reprint of the 1974 original.
%
%\bibitem{amigosjoao}
% V. H. Jorge PÃ©rez, M. J. Saia,\
% {\em Euler Obstruction, Polar multiplicities and Equisingularity of Map Germs in $\mathcal{O}_{(n,p)},\;n<p$,} Internat. J. Math. 17 (2006), No. 8, 887-903.
%
\bibitem{Grego}
  K. Kourliouros,\
  {\em The Milnor-Palamodov Theorem for Functions on Isolated Hypersurface Singularities.} Bull Braz Math Soc, New Series (2021), no. 2, 405--413.

\bibitem{perez saia}
 V. H. Jorge P\' erez, M. J. Saia,\
 {\em Euler obstruction, polar multiplicities and equisingularity of map germs in $\mathcal{O}_{(n,p)},\;n<p$,} Internat. J. Math. 17 (2006), No. 8, 887--903.


\bibitem{le tessier}
  D. T. L\^ e, B. Tessier,\
  {\em Varietes polaires locales et classes de Chern des varietes singulieres,} Ann. of Math. 114 (1981), 457--491.

\bibitem{nunoballesteros oreficeokamoto limapereira tomazela}
B. K.  Lima-Pereira, J. J. Nu\~no-Ballesteros, B. Or\'efice-Okamoto,  J. N. Tomazella,\
{\em The Bruce--Roberts number of a function on a hypersurface with isolated singularity.} Q. J. Math. { 71} (2020), no. 3, 1049--1063. 

\bibitem{segundo artigo} 
  B. K.  Lima-Pereira, J. J. Nu\~no-Ballesteros, B. Or\' efice-Okamoto, J. N. Tomazella,\
{\em The relative Bruce-Roberts number of a function on a hypersurface.} Proc. Edinb. Math. Soc. (2) 64 (2021), no. 3, 662--674. 


\bibitem{nuno mond}
D. Mond, J. J. Nu\~ no Ballesteros,\
{\em Singularities of mappings}, volume 357 of Grundlehren der mathematischen Wissenschaften. Springer, Cham, 2020.



  

%\bibitem{looijenga}
%  E. J. N. Looijenga,\
%{\em Isolated singular points on complete intersections}, London
%Mathematical Society Lecture Note Series, 77. Cambridge University Press
%(1984).
%\bibitem{eulerobstruction}
%R. D. MacPherson,\
%{\em Chern classes for singular algebraic varieties}, Ann. of Math. 100 (1974), 423?432.
%
%\bibitem{Milnor} 
%  J. Milnor,\
%{\em Singular Points of Complex Hypersurfaces}, Annals of Mathematical Studies 61, Princeton University Press, Princeton, 1968.
%
%\bibitem{not2}
%J. J. Nu\~no-Ballesteros, B. Or\'efice-Okamoto, J. N. Tomazella,\
%  {\em The Vanishing Euler Characteristic of an Isolated Determinantal Singularity}, Israel Journal of Mathematics, v. 224, p. 505-512, 2018.
%
\bibitem{Orefice}
  J. J. Nu\~no Ballesteros, B. Or\'efice-Okamoto, J. N. Tomazella,\
  {\em The Bruce-Roberts number of a function on a weighted homogeneous hypersurface,} Q. J. Math. 64 (2013), no. 1, 269-280.

%\bibitem{tesebruna}
%  B. OrÃ©fice-Okamoto,\
%{\em O NÃºmero de Milnor de uma Singularidade Isolada}, Thesis, SÃ£o Carlos, 2010
%
%\bibitem{pellikaanh}
%  G. R. Pellikaanh,\
% {\em Hypersurface Singularities and Resolutions of Jacobi Modules}, Thesis, Utrecht, 1985.
% 
%\bibitem{saito}
% K. Saito,\
% {\em Theory of logarithmic differential forms and logarithmic vector fields}, J. Fac. Sci. Univ.
% Tokyo Sect. 1A Math. 27 (1980), 265--291.
% 
%\bibitem{STV}
%J. Seade, M. Tibar and A. Verjovsky,\
%{\em  Milnor numbers and Euler obstruction}, Bull. Braz. Math.
%Soc. (N.S.) 36 (2005), 275--283.
%
\bibitem{Tajima}
  S. Tajima,\
  {\em On Polar Varieties, Logarithmic Vector Fields and Holonomic D-modules}, Recent development of micro-local analysis for the theory of asymptotic analysis, 41--51, RIMS Kôkyûroku Bessatsu, B40, Res. Inst. Math. Sci. (RIMS), Kyoto, 2013.
%
%\bibitem{vosegaard}
%H. Vosegaard,\
%{\em A characterization of quasi-homogeneous complete intersection singularities}, J. Algebraic Geom. { 11} (2002), no. 3, 581--597. 
\bibitem{wahl}
  J. M. Wahl,\
{\em Derivations, automorphisms and deformations on quasi-homogeneous singularity,} Singularities, Part 2 (Arcata, Calif., 1981), 613–624, Proc. Sympos. Pure Math., 40, Amer. Math. Soc., Providence, RI, 1983.








\end{thebibliography}
\end{document}